\documentclass[reqno,12pt]{amsart} 
\usepackage{vmargin}
\setpapersize{A4}

\usepackage{amsmath} 

\usepackage{amsfonts}   

\usepackage{amssymb}      

\usepackage{eufrak}      

\usepackage{amscd}      

\usepackage{amsthm}      
   
\usepackage{epsfig}      

\usepackage{amstext}      

\usepackage[all,line,dvips]{xy} 
\CompileMatrices 

\newcommand{\id}{\operatorname{id}} 
\newcommand{\Aut}{\operatorname{Aut}}

\newcommand{\Span}{\operatorname{Span}}

 \newcommand{\supp}{\operatorname{supp}}

\newcommand{\alg}{\operatorname{alg}}

   \theoremstyle{plain}
   \newtheorem{thm}{Theorem}[section]
   \newtheorem{prop}[thm]{Proposition}
   \newtheorem{lemma}[thm]{Lemma}  
   \newtheorem{cor}[thm]{Corollary}
   \theoremstyle{definition}
   
   \newtheorem{defn}[thm]{Definition}
   \newtheorem{example}[thm]{Example}
   \theoremstyle{remark}
   
   \newtheorem{remark}[thm]{Remark}

\usepackage{graphicx}

   \numberwithin{equation}{section}

        \date{\today}

\title[Local homeomorphisms]{On the $C^*$-algebra of a locally injective
  surjection and its KMS states}
\author{Klaus Thomsen}

\date{\today}

\email{matkt@imf.au.dk}
\address{Institut for matematiske fag, Ny Munkegade, 8000 Aarhus C, Denmark}

\begin{document}

\maketitle

\section{Introduction}\label{sec1}

In \cite{Th} the construction of a $C^*$-algebra from an \'etale
groupoid, as introduced by J.Renault in \cite{Re1}, was
generalized to a larger class of locally compact groupoids called
semi-\'etale groupoids, where the
range and source maps are locally injective, but not necessarily open. The
main purpose with the generalization was to make the powerful
techniques for \'etale groupoids available to the study of
dynamical systems via the groupoid constructed in increasing generality
by Renault, Deaconu and Anantharaman-Delaroche, \cite{Re1}, \cite{D},
\cite{A}, also when the underlying map is not open. In particular, as shown in \cite{Th} this makes it possible to handle general
(one-sided) subshifts. 

One of the intriguing connections between dynamical systems and
$C^*$-algebras is the relation between the termodynamical formalism of
Ruelle, as described in \cite{Ru}, and quantum statistical mechanics,
as described in \cite{BR}. One relation between these formalisms is
very concrete and direct and manifests itself in almost all of the
$C^*$-algebraic settings of quantum statistical mechanics through a bijective
correspondance between KMS states and measures fixed by a dual
Ruelle operator. This relation is implicit in the work of J. Renault,
\cite{Re1} and \cite{Re2}, and has been developed further by R. Exel,
\cite{E}. By using this correspondance Kumjian and Renault,
\cite{KR}, were able to use Walters results, \cite{W2}, on the convergence of the
Ruelle operator to extend most results on the existence and
uniqueness of KMS states for the generalized gauge
actions on Cuntz-Krieger algebras which has been one of the favourite
models in quantum statistical mechanics.

The main purpose with the present work is to show that there is a
canonical way to pass from a locally injective continuous surjection
to a local homeomorphism in such a way that the $C^*$-algebras of the
corresponding groupoids, one of them defined as in \cite{Th}, are
isomorphic. The construction is a generalization of W. Kriegers construction of a canonical extension for a sofic shift, \cite{Kr1}, \cite{Kr2},
now known as the left Krieger cover. The canonical local homeomorphic
extension of a general locally injective surjection which we construct
is
undoubtedly useful for other purposes, and it seems to deserve a more
thorough investigation. Here we use it to investigate the KMS states of
the generalized gauge actions. In fact, we restrict our considerations
even further by focusing only on the possible values of the inverse
temperature $\beta$ for such KMS states. The results we obtain give
bounds on the possible $\beta$-values and ensure the existence of KMS
states under mild conditions on the potential function. We depart from
the work of Exel in \cite{E} and the main tool to prove existence of KMS
states is a method 
developed by Matsumoto, Watatani and Yoshida in \cite{MWY} and
Pinzari, Watatani and Yonetani in \cite{PWY}. Concerning bounds on the
possible $\beta$-values of KMS states the main novelty is the
observation that it is not so much the entropy of the map which
provides the bounds but rather the exponential growth rate of the
number of pre-images. The relevant entity is thus an invariant $h_m$
which was introduced by Hurley in \cite{Hu} and studied further in
\cite{FFN}, among others. For forward expansive maps the invariant of Hurley
is equal to the topological entropy, but in generally it is smaller
than the topological entropy. The invariant of Hurley controls the
existence of KMS states completely when the potential function is
strictly positive or strictly negative: For such potential functions
there is a KMS-state if and only if $h_m$ is not zero. We refer to Section \ref{KMSsection} for
more details on our results on KMS states.


\section{Recap about $C^*_r\left(\Gamma_{\varphi}\right)$}\label{sec2}

Let $X$ be a locally compact Hausdorff space and $\varphi : X \to X$ a
continuous map. We assume that $\varphi$ is locally injective, meaning
that there is a basis for the topology of $X$ consisting of sets on
which $\varphi$ is injective. Set
$$
\Gamma_{\varphi} = \left\{ (x,k,y) \in X \times \mathbb Z \times X : \ \exists
  a,b \in \mathbb N, \ k = a-b, \ \varphi^a(x) = \varphi^b(y) \right\} .
$$
This is a groupoid with the set of composable pairs being
$$
\Gamma_{\varphi}^{(2)} \ =  \ \left\{\left((x,k,y), (x',k',y')\right) \in \Gamma_{\varphi} \times
  \Gamma_{\varphi} : \ y = x'\right\}.
$$
The multiplication and inversion are given by 
$$
(x,k,y)(y,k',y') = (x,k+k',y') \ \text{and}  \ (x,k,y)^{-1} = (y,-k,x)
.
$$
To turn $\Gamma_{\varphi}$ into a locally compact topological groupoid, fix $k \in \mathbb Z$. For each $n \in \mathbb N$ such that
$n+k \geq 0$, set
$$
{\Gamma_{\varphi}}(k,n) = \left\{ \left(x,l, y\right) \in X \times \mathbb
  Z \times X: \ l =k, \ \varphi^{k+i}(x) = \varphi^i(y), \ i \geq n \right\} .
$$
This is a closed subset of the topological product $X \times \mathbb Z
\times X$ and hence a locally compact Hausdorff space in the relative
topology.
Since $\varphi$ is locally injective $\Gamma_{\varphi}(k,n)$ is an open subset of
$\Gamma_{\varphi}(k,n+1)$ and hence the union
$$
{\Gamma_{\varphi}}(k) = \bigcup_{n \geq -k} {\Gamma_{\varphi}}(k,n) 
$$
is a locally compact Hausdorff space in the inductive limit topology. The disjoint union
$$
\Gamma_{\varphi} = \bigcup_{k \in \mathbb Z} {\Gamma_{\varphi}}(k)
$$
is then a locally compact Hausdorff space in the topology where each
${\Gamma_{\varphi}}(k)$ is an open and closed set. In fact, as is easily verified, $\Gamma_{\varphi}$ is a locally
compact groupoid in the sense of \cite{Re1}. Note that the unit space $\Gamma_{\varphi}^0$ of
$\Gamma_{\varphi}$ equals $X$ via the identification $x \mapsto (x,0,x)$. The local
injectivity of $\varphi$ ensures that the range map $r(x,k,y) = x$ is
locally injective, i.e. $\Gamma_{\varphi}$
is semi \'etale. We can therefore define the corresponding
$C^*$-algebra $C^*_r\left(\Gamma_{\varphi}\right)$ as in
\cite{Th}. Briefly $C^*_r\left(\Gamma_{\varphi}\right)$ is the
completion of the $*$-algebra $\alg^* \Gamma_{\varphi}$ generated by the continuous and compactly
supported function on $\Gamma_{\varphi}$ under the convolution product
$$
f \star g (x,k,y) = \sum_{z, m+n = k} f(x,n,z)g(z,m,y) ,
$$
and the involution
$$
f^*(x,k,y) = \overline{f(y,-k,x)} .
$$  
The elements of $\alg^* \Gamma_{\varphi}$ are all bounded and of
compact support, but not necessarily continuous. The elements of $\alg^*
\Gamma_{\varphi}$ whose supports are contained in the unit space,
identified with $X$ as it is, generate under the completion an abelian
$C^*$-algebra $D_{\Gamma_{\varphi}}$ which contains $C_0(X)$ and consists of bounded
functions vanishing at infinity. The restriction map extends to a
conditional expectation $P_{\Gamma_{\varphi}} :
C^*_r\left(\Gamma_{\varphi}\right) \to D_{\Gamma_{\varphi}}$.

Let us now restrict the attention to the case where $X$ is
compact and metrizable. One of the results from \cite{Th} is that
$C^*_r\left(\Gamma_{\varphi}\right)$ can then be realized as a crossed
$C^*_r\left(R_{\varphi}\right) \times_{\widehat{\varphi}} \mathbb N$
in the sense of Paschke,
where $C^*_r\left(R_{\varphi}\right)$ is the $C^*$-subalgebra of
$C^*_r\left(\Gamma_{\varphi}\right)$ generated by
$C_c\left(\Gamma_{\varphi}(0)\right)$ and $\widehat{\varphi}$ is the
endomorphism of $C^*_r\left(\Gamma_{\varphi}\right)$ given by conjugation
with the isometry $V_{\varphi}$, where 
$$
V_{\varphi}(x,k,y) = \begin{cases} m(x)^{-\frac{1}{2}} \ & \
  \text{when $k =1$ and $y = \varphi(x)$} \\ 0 \ & \ \text{otherwise.} \end{cases}
$$
The function $m : X \to \mathbb N$ which enters here is also going to
play an important role in the present paper and it is equal to $m = N \circ \varphi$, with
$$
N(x) = \# \varphi^{-1}(x) .
$$
While this crossed product description is useful for several purposes,
including the calculation of the $K$-theory groups of
$C^*_r\left(\Gamma_{\varphi}\right)$, it is going to be instrumental
here to relate to a crossed product description in the sense of Exel, \cite{E}.

\section{$C^*_r\left(\Gamma_{\varphi}\right)$ as a crossed product in
  the sense of Exel}\label{sec3}

Let $f \in D_{\Gamma_{\varphi}}$. Then 
$P_{\Gamma_{\varphi}}\left(V_{\varphi}fV_{\varphi}^*\right)(x) =
m(x)^{-1}f\left(\varphi(x)\right)$.
Since $m \in D_{\Gamma_{\varphi}}$ this
shows that $f \circ \varphi \in D_{\Gamma_{\varphi}}$. We can
therefore define a $*$-endomorphism 
$\alpha_{\varphi}$ of $D_{\Gamma_{\varphi}}$
such that 
\begin{equation}\label{alphavarphi}
\alpha_{\varphi}(f) = f \circ \varphi.
\end{equation}
Note that $\alpha_{\varphi}$
is unital, and injective since $\varphi$ is surjective. Let $f \in
D_{\Gamma_{\varphi}}$, and let $1_{\Gamma_{\varphi}(1,0)}$ be the
characteristic function of the open and compact subset
$\Gamma_{\varphi}(1,0)$ of $\Gamma_{\varphi}$. Then
$1_{\Gamma_{\varphi}(1,0)}^*f1_{\Gamma_{\varphi}(1,0)} \in
D_{\Gamma_{\varphi}}$ and
\begin{equation}\label{tauborineman}
1_{\Gamma_{\varphi}(1,0)}^*f1_{\Gamma_{\varphi}(1,0)}(x) = \sum_{z \in
  \varphi^{-1}(x)} f(z) .
\end{equation}
Hence the function $X \ni x \mapsto \sum_{z \in
  \varphi^{-1}(x)} f(z)$ is in $D_{\Gamma_{\varphi}}$. In particular,
the function
$$
N(x) = \# \varphi^{-1}(x) = \sum_{z \in
  \varphi^{-1}(x)} 1
$$
is in $D_{\Gamma_{\varphi}}$. This allows us to define $\mathcal L_{\varphi} :
D_{\Gamma_{\varphi}} \to D_{\Gamma_{\varphi}}$ such that
$$
\mathcal L_{\varphi}(f)(x) = N(x)^{-1} \sum_{z \in
  \varphi^{-1}(x)} f(z) .
$$
$\mathcal L_{\varphi}$ is a unital positive linear map and 
$\mathcal L_{\varphi}\left(f \alpha_{\varphi}(g)\right) = \mathcal
L_{\varphi}(f)g$ for all $f,g \in
D_{\Gamma_{\varphi}}$. Hence $\mathcal L_{\varphi}$ is a transfer operator in
the sense of Exel, cf. \cite{E} and \cite{EV}, so that the crossed
product 
$$
D_{\Gamma_{\varphi}} \rtimes_{\alpha_{\varphi}, \mathcal L_{\varphi}} \mathbb N
$$
is defined. Observe that $\mathcal L_{\varphi}$ is faithful and that
the Standing Hypotheses of \cite{EV}, Hypotheses 3.1, are all
satisfied.

The following result generalizes Theorem 9.2 in \cite{EV}, and to some
extend also Theorem 4.18 of \cite{Th}.

\begin{thm}\label{exeliso} There is a $*$-isomorphism $D_{\Gamma_{\varphi}} \rtimes_{\alpha_{\varphi}, \mathcal L_{\varphi}} \mathbb N \to
 C^*_r\left(\Gamma_{\varphi}\right)$ which is the identity on
  $D_{\Gamma_{\varphi}}$ and takes the isometry $S$ of Exel
  (cf. \cite{E}) to the isometry $V_{\varphi} \in
  C^*_r\left(\Gamma_{\varphi}\right)$.
\begin{proof} Since $\varphi$ is locally injective there is a partition of unity
$\{b_i\}_{i=1}^k$ in $C(X) \subseteq D_{\Gamma_{\varphi}}$ such that
$\varphi$ is injective on $\supp b_i$ for each $i$. It is then
straightforward to check that
$$
f = \sum_{i=1}^k \left(b_i m \right)^{\frac{1}{2}}
\alpha_{\varphi} \circ \mathcal L_{\varphi}\left(\left(b_i m\right)^{\frac{1}{2}} f\right)
$$
for all $f \in D_{\Gamma_{\varphi}}$, so that $\left\{\left(b_i m\right)^{\frac{1}{2}}\right\}_{i=1}^k$ is a
quasi-basis for the conditional expectation $\alpha_{\varphi} \circ
\mathcal L_{\varphi}$ of $D_{\Gamma_{\varphi}}$ onto
$\alpha_{\varphi}\left(D_{\Gamma_{\varphi}}\right)$ in the sense of \cite{EV}. It
is also straightforward to check that $V_{\varphi}f = \alpha_{\varphi}(f)V_{\varphi}$
and $V_{\varphi}^*fV_{\varphi} = \mathcal L_{\varphi}(f)$ 
for all $f \in
  D_{\Gamma_{\varphi}}$. Furthermore,  
$$
1 = \sum_{i=1}^k
\left(b_im \right)^{\frac{1}{2}}V_{\varphi}V_{\varphi}^*\left(b_i m\right)^{\frac{1}{2}} .
$$
It follows therefore from Corollary 7.2 of \cite{EV} that there is a
$*$-homomorphism $\rho : D_{\Gamma_{\varphi}} \rtimes_{\alpha_{\varphi},
  \mathcal L_{\varphi}}  \mathbb N\to
  C^*_r\left(\Gamma_{\varphi}\right)$ which is the identity on
  $D_{\Gamma_{\varphi}}$ and takes the isometry $S$ to the isometry $V_{\varphi} \in
  C^*_r\left(\Gamma_{\varphi}\right)$. To see that $\rho$ is
  surjective we must show that $C^*_r\left(\Gamma_{\varphi}\right)$ is
  generated by $D_{\Gamma_{\varphi}}$ and $V_{\varphi}$. From the expresssion
  for $V_{\varphi}^n\left(V_{\varphi}^*\right)^n$ given in the proof of Theorem 4.8 of
  \cite{Th}, combined with Corollary 4.5 from \cite{Th}, it follows
  that the $C^*$-algebra generated by $V_{\varphi}$ and $D_{\Gamma_{\varphi}}$
  contains the characteristic function $1_{R\left(\varphi^n\right)}$
  for each $n$. It follows then that it contains
\begin{equation}\label{prodfunc}
C(X) \star 1_{R\left(\varphi^n\right)} \star C(X)  
\end{equation}
since $C(X) \subseteq D_{\Gamma_{\varphi}}$. Among the functions in (\ref{prodfunc}) are the elements of
$C\left(R\left(\varphi^n\right)\right)$ which are restrictions to
$R\left(\varphi^n\right)$ of
product type functions, $X \times X \ni (x,y) \mapsto f(x)g(y)$, with
$f,g \in C(X)$. These functions generate $C(X \times X)$ and their
restriction generate $C\left(R\left(\varphi^n\right)\right)$ so it
follows that the $C^*$-algebra generated by $V_{\varphi}$ and $D_{\Gamma_{\varphi}}$
  contains $C\left(R\left(\varphi^n\right)\right)$ for each $n$. Since
$$
C^*_r\left(R_{\varphi}\right) = \overline{\bigcup_n
  C\left(R\left(\varphi^n\right)\right)}
$$
we conclude from Theorem 4.6 of \cite{Th} that it
  coincides with $C^*_r\left(\Gamma_{\varphi}\right)$, proving that
  $\rho$ is surjective. Finally, it follows from Theorem 4.2 of
  \cite{EV} that $\rho$ is injective since the gauge action on
  $C^*_r\left(\Gamma_{\varphi}\right)$ can serve as the required $\mathbb
  T$-action.
\end{proof}
\end{thm}

\section{A canonical local homeomorphism extending $(X,\varphi)$}\label{sec4}

In this section we show that the continuous map $\psi$
from the Gelfand spectrum of
$D_{\Gamma_{\varphi}}$ to itself which corresponds to the
endomorphism (\ref{alphavarphi}) of $D_{\Gamma_{\varphi}}$ is a local homeomorphism
and that the corresponding dynamical system is a canonical extension
of $(X,\varphi)$. The proof is based on the well-known contravariant
equivalence between compact Hausdorff spaces and unital abelian
$C^*$-algebras.

To simplify notation, set $D = D_{\Gamma_{\varphi}}$ and let
$\widehat{D}$ be the Gelfand spectrum of
$D_{\Gamma_{\varphi}}$. Recall that $\widehat{D}$ consists of the unital
$*$-homomorphisms $c: D \to \mathbb C$, also known as the
\emph{characters} of $D$. $\widehat{D}$ is closed in the
weak*-topology of the unit
ball in the dual space $D^*$ of $D$ and obtains in this way a compact
topology. Since $X$ is compact and metrizable it follows that $D$ is
separable and it follows that also $\widehat{D}$ is metrizable. Finally, recall that every
element $d \in D$ becomes a continuous function on $\widehat{D}$ in
the natural way; viz. $d(c) = c(d)$, and this recipe gives rise to an
(isometric) $*$-isomorphism between $D$ and
$C(\widehat{D})$ which we suppress in the notation by
simply identifying $D$ and $C(\widehat{D})$ whenever it is
convenient.

There is a map $\pi : \widehat{D} \to X$
arising from the fact that every character of $C(X)$ comes from
evaluation at point in $X$: Given a character $c \in \widehat{D}$ of
$D$ there is a unique point $\pi(c) \in X$ such that
$c(f) = f\left(\pi(c)\right) $
for all $f \in C(X)$. Note that $\pi$ is continuous. We define $\psi : \widehat{D} \to \widehat{D}$ such that
$\psi(c)(g) = c\left(g \circ \varphi\right) $
for all $g \in D$. It follows straightforwardly from the definition of
the topology of $\widehat{D}$ that $\psi$ is continuous. Hence
$\left(\widehat{D},\psi\right)$ is a dynamical system. Note that
$$
f\left(\left(\varphi \circ \pi\right(c)\right) = f \circ
\varphi\left(\pi(c)\right) = c\left(f \circ \varphi\right) =\psi(c)(f)
= f\left(\pi \circ \psi(c)\right)
$$
for all $f \in C(X)$, proving that $\pi : \left(X,\varphi\right) \to
(\widehat{D}, \psi)$ is equivariant. Define $\iota : X \to \widehat{D}$ by $\iota(x) = c_x \in
\widehat{D}$ where $c_x$ is the character defined such that
$c_x(g) = g(x)$
for all $g \in D$. Since $g\left(\psi \circ \iota(x)\right) =
c_x\left( g \circ \varphi\right) = g\left(\varphi(x)\right) = c_{\varphi(x)}(g)$ we see
that also $\iota : (X,\varphi) \to (\widehat{D}, \psi)$ is
equivariant. Furthermore    
$\pi \circ \iota(x) = x$
for all $x \in X$, proving that $\iota$ is injective and $\pi$
surjective. Note, however, that $\iota$ is generally not
continuous. Since $g \in D, c_x(g) = 0 \ \forall x \in X \
\Rightarrow \ g = 0$, the range $\iota(X)$ of $\iota$ is dense in
$\widehat{D}$. 

It is
evident that the construction of $(\widehat{D}, \psi)$  is
canonical in the following sense: If $\varphi' : X' \to X'$ is another locally
injective surjection of a compact Hausdorff space $X'$, then a conjugacy from
$(X,\varphi)$ to $(X',\varphi')$ induces a conjugacy from
$(\widehat{D},\psi)$ to $(\widehat{D'},\psi')$
which extends the given conjugacy in the sense
that the diagram
\begin{equation*}
\begin{xymatrix}{
{\widehat{D}} \ar[d]_-{\pi} \ar[r] & {\widehat{D'}} \ar[d]^-{\pi'} \\
X \ar[r] & X' }
\end{xymatrix}
\end{equation*}
commutes.

It remains now only to establish the following

\begin{prop}\label{prop} $\psi$ is a surjective local homeomorphism.
\begin{proof} $\psi$ is locally injective: Let $c \in \widehat{D}$ and
  set $z = \pi(\psi(c)) = \varphi\left(\pi(c)\right)$. By Lemma 3.6 of \cite{Th} there is an open
  neighborhood $U$ of $z$ and open sets $V_i, i = 1,2, \dots, j$,
  where $j =\#
  \varphi^{-1}(z)$, such that
\begin{enumerate}
\item[1)] $\varphi^{-1}\left(\overline{U}\right) \subseteq V_1 \cup
  V_2 \cup \dots \cup V_j$,
\item[2)] $\overline{V_i} \cap \overline{V_{i'}} = \emptyset$ when $i
  \neq i'$, and
\item[3)] $\varphi$ is injective on $\overline{V_i}$ for each $i$.
\end{enumerate}
Without loss of generality we may assume that $\pi(c) \in V_1$. Let $h,H \in C(X)$ be
  such that $0 \leq h \leq 1$, $h\left(\pi(c)\right) = 1$,
  $\varphi\left(\supp h\right) \subseteq U$, $Hh = h$ and $\supp H
  \subseteq \overline{V_1}$. Set 
$$
W = \left\{ c' \in \widehat{D} : \ c'(h) > 0 \right\} ;
$$
clearly an open subset of $\widehat{D}$. To show that $c \in W$ we
choose a sequence $\{z_k\}$ in $X$ such that $\lim_k
\iota\left(z_k\right) = c$. Then $\pi(c) = \lim_k \pi \circ
\iota\left(z_k\right) = \lim_k z_k$ so that 
$$
c(h) = \lim_k \iota\left(z_k\right)(h) = \lim_k h\left(z_k\right) =
h\left(\pi(c)\right) = 1.
$$
$W$ is therefore an open neighborhood of $c$ in $\widehat{D}$. To show
that $\psi$ is injective on $W$, let $c',c'' \in W$ and choose
sequences $\left\{z'_k\right\}$ and $\left\{z''_k\right\}$ in $X$ such
that $\lim_k \iota\left(z'_k\right) = c'$ and $\lim_k
\iota\left(z''_k\right) = c''$. Since
$$
\lim_k h\left(z'_k\right) = \lim_k \iota\left(z'_k\right)(h) = c'(h) >
0,
$$
it follows that $h\left(z'_k\right) > 0$ for all large $k$. Hence
$\varphi\left(z'_k\right) \in U$,
$H\left(z'_k\right) = 1$ and $z'_k \in \overline{V_1}$ for all large
$k$.
It follows that
\begin{equation*}
\begin{split}
&\psi(c')\left(\sum_{v \in \varphi^{-1 }(\cdot)} fH(v)\right) = \lim_k
\iota \circ \varphi\left(z'_k\right)\left(\sum_{v \in \varphi^{-1
    }(\cdot)} fH(v)\right) \\
& = \lim_k \sum_{v \in \varphi^{-1}(\varphi(z'_k))} fH(v) = \lim_k f(z'_k) =
c'(f)
\end{split}
\end{equation*}
for all $f \in D$. Similarly, 
$$ 
\psi(c'')\left(\sum_{v \in \varphi^{-1 }(\cdot)} fH(v)\right) = c''(f)
$$
for all $f \in D$. It follows that $\psi(c') = \psi(c'') \ \Rightarrow
\ c' = c''$, proving that $\psi$ is injective on $W$.

\smallskip

$\psi$ is open: Let $ f \in D$ be a non-negative function and set
$$
V = \left\{ c \in \widehat{D} : \ c(f) > 0 \right\}
.
$$
It suffices to show that $\psi(V)$ is open in $\widehat{D}$, so we
consider an element $c
\in V$, and set  
$$
W = \left\{ c' \in \widehat{D} : \ c'\left(  \sum_{v \in
  \varphi^{-1}(\cdot)} f(v) \right) > \frac{c(f)}{2}  \right\} .
$$
Let $\{z_k\}$ be a sequence in $X$ such that $\lim_k
\iota\left(z_k\right) = c$ and note that
\begin{equation*}
\begin{split}
&\psi(c)\left(  \sum_{v \in
  \varphi^{-1}(\cdot)} f(v) \right) = \lim_k \iota\left(\varphi(z_k)\right)\left(\sum_{v \in
  \varphi^{-1}(\cdot)} f(v)\right) \\
&= \lim_k \sum_{v \in
  \varphi^{-1}(\varphi(z_k))} f(v) 
 \geq \lim_k f(z_k) = \lim_k \iota(z_k)(f) = c(f) > \frac{c(f)}{2}.
\end{split}
\end{equation*}
It follows that $W$ is an open neighborhood of $\psi(c)$. It suffices
therefore to show that $W \subseteq \psi(V)$. Let $c' \in W$ and choose
a sequence $\{z'_k\}$ in $X$ such that $\lim_{k \to \infty} \iota(z'_k) = c'$
in $\widehat{D}$. For all large $k$, 
$$
\sum_{v \in
  \varphi^{-1}(z'_k)} f(v) = \iota\left(z'_k\right)\left( \sum_{v \in
    \varphi^{-1}(\cdot)} f \right)  > \frac{c(f)}{2}, 
$$
so for all large $k$ there are elements $v_k \in \varphi^{-1}(z'_k)$ such that
$f(v_k) \geq \frac{c(f)}{2M}$, where $M = \max_{x \in X} \#
\varphi^{-1}(x)$. Let $c''$ be a condensation point in $\widehat{D}$ of
the sequence $\left\{\iota(v_k)\right\}$. For the corresponding
subsequence $\left\{v_{k_i}\right\}$ we find that $\psi(c'') = \lim_i
\varphi\left(v_{k_i}\right) = \lim_i z'_{k_i} = c'$. Since
$$
c''(f) = \lim_i f\left(v_{k_i}\right) \geq  \frac{c(f)}{2M} > 0,
$$ 
it follows that $c'' \in V$, proving that $W \subseteq \psi(V)$.

\smallskip

$\psi$ is surjective: If $\psi(\widehat{D}) \neq \widehat{D}$, there
is an element $f \in D$ such that $f \neq 0$ and $f \geq 0$, while $\psi(c)(f) = 0$
for all $c \in \widehat{D}$. Since $\psi(c)(f) = c\left(f \circ
  \varphi\right)$ it follows that $f \circ \varphi = 0$. This is
impossible since $f \neq 0$ and $\varphi$ is surjective.

\end{proof}
\end{prop}

The dynamical system $(\widehat{D},\psi)$ will be called
\emph{the canonical local homeomorphic extension} of $(X,\varphi)$.
It can be shown that $(\widehat{D},\psi)$ is the left
Krieger cover of $(X,\varphi)$ when $(X,\varphi)$ is a one-sided sofic
shift.

\section{Isomorphism of the $C^*$-algebras
  $C^*_r\left(\Gamma_{\varphi}\right)$ and $C^*_r\left(\Gamma_{\psi}\right)$}\label{sec5} 

Since $\psi$ is a local homeomorphism the $C^*$-algebras
$C^*_r\left(R_{\psi}\right)$ and $C^*_r\left(\Gamma_{\psi}\right)$
coincide with the one considered in \cite{A}. In particular, the
abelian $C^*$-algebra $D_{\Gamma_{\psi}}$ is equal to $C(\widehat{D})
= D_{\Gamma_{\varphi}}$. In this section we show that this
identification, $D_{\Gamma_{\varphi}} = D_{\Gamma_{\psi}}$, is the restriction of an isomorphism between
$C^*_r\left(\Gamma_{\varphi}\right)$ and $C^*_r\left(\Gamma_{\psi}\right)$.

As above we let $N \in D$ be the function
$N(x) = \# \varphi^{-1}(x)$,
and set 
$
m = N \circ \varphi$.

\begin{lemma}\label{N} $c(N) = \# \psi^{-1}(c)$ for all $c \in
  \widehat{D}$.
\begin{proof} For any $f \in D$, let $I(f)$ denote the function
$$
I(f)(x) = \sum_{v \in \varphi^{-1}(x)} f(v) .
$$
It follows from (\ref{tauborineman}) that $I(f) \in D$. Let $c \in
\widehat{D}$ and let $\{z_k\}$ be a sequence in $X$ such that $\lim_k
  \iota(z_k) = c$. Set $z = \pi(c)$, and let be $U,V_1,V_2, \dots,V_j$ as
  in Lemma 3.6 of \cite{Th}, i.e. 1)-3) from the proof of Proposition
  \ref{prop} hold. Since $\lim_k N(z_k) = c(N)$ we can
  assume that $N(z_k) = c(N)$ for all $k$, and since $\lim_k z_k =
  \lim_k \pi \circ \iota\left(z_k\right) = z$
  in $X$ we can assume that $z_k \in U$ for all $k$. Choose functions $h_i, H_i
\in C_c(X), i = 1,2, \dots, j$, such that $0 \leq h_i \leq 1$,
$h_i\left(w_i\right) = 1$, where $w_i = V_i \cap \varphi^{-1}(z)$,
  $\varphi\left(\supp h_i\right) \subseteq U$, $H_ih_i = h_i$ and $\supp H_i
  \subseteq \overline{V_i}$ for all $i$.

Observe that $c(N) \leq j$ and set
$$
g_F = \prod_{i \in F} I(h_i) \in D
$$  
for every subset $F \subseteq \{1,2, \dots, j\}$ with $c(N)$
elements. For all sufficiently large $k$ there is a subset $F \subseteq \{1,2, \dots,
j\}$ with $c(N)$ elements such that $g_F(z_k) \geq \frac{1}{2}$. Indeed,
since $N(z_k) = c(N)$ there is for each $k$ a subset $F_k \subseteq
\left\{1,2, \dots, j\right\}$ with $c(N)$ elements and elements $v_k^i
\in V_i, \ i \in F_k$, such that
$\varphi^{-1}\left(z_k\right) = \left\{ v_k^i : \ i \in
  F_k\right\}$. When $g_{F_k}\left(z_k\right) <\frac{1}{2}$ there must
be at least one $i_k \in F_k$ for which 
$$
h_{i_k}\left(v_k^{i_k}\right) <
\left(\frac{1}{2}\right)^{\frac{1}{c(N)}} .
$$
Hence, if $g_{F_k}\left(z_k\right) <\frac{1}{2}$ for infinitely many
$k$, a condensation point of the sequence $\left\{v_k^{i_k}\right\}$
would give us, for some $i' \in \left\{1,2,\dots, j\right\}$,  a point
in $\overline{V_{i'}} \cap \varphi^{-1}(z)$ other than $w_{i'}$,
contradicting property 3) of the $V_i$'s. Hence
$g_{F_k}\left(z_k\right) \geq \frac{1}{2}$ for all sufficiently large
$k$. Since there are only
finitely many subsets of $\{1,2,\dots, j\}$ we can pass to a
subsequence of $\{z_k\}$ to arrange that the same subset $F'$ works for
all $k$, i.e. that 
\begin{equation}\label{equalone}
g_{F'}(z_k) \geq \frac{1}{2}
\end{equation}
for all $k$. Since $N(z_k) = c(N)
= \# F'$
this implies that   
$$
\varphi^{-1}(z_k) = \left\{v^i_k : \ i \in F' \right\}
$$
for some (unique) elements $v_k^i \in V_i, i \in F'$. 
 For each $i$,
let $c^i$ be a condensation point of $\{\iota\left(v^i_k\right)\}$ in
$\widehat{D}$. Then $\psi(c^i) = \lim_k
\psi\left(\iota\left(v^i_k\right)\right) = \lim_k
\iota\left(z_k\right) = c$ for all
$i$. Since $c^i\left(h_{i'}\right) = \lim_k h_{i'}\left(v^i_k\right) \neq 0$ if and 
only if $i=i'$ for $i,i' \in F'$, we conclude that $c^i \neq
c^{i'}$ when $i \neq i'$, proving that $\# \psi^{-1}(c) \geq c(N)$.

 As shown in the proof
  of Proposition \ref{prop}, $\psi$ is injective on 
$$
W_i = \left\{ c' \in \widehat{D} : \ c'\left(h_i\right) > 0 \right\} .
$$ 
To show that $\# \psi^{-1}(c) \leq N(c)$ it suffices therefore to show
that every element $c''$ of $\psi^{-1}(c)$ is contained in $W_i$ for
some $i \in F'$. To this end we pick a sequence $\{y_k\}$ in $X$ such
that $\lim_k \iota\left(y_k\right) = c''$ in $\widehat{D}$. Set $z'_k = \varphi(y_k)$ and
note that $\lim_k \iota\left(z'_k\right) = \psi(c'') =c$ while $\lim_k
z'_k = \lim_k \pi \circ \psi \circ \iota\left(y_k\right) = \lim_k \pi
\circ \psi\left(c''\right) = z$. In particular,
\begin{equation}\label{equaltwo}
N(z'_k) = c( N)
\end{equation} 
and 
\begin{equation}\label{equalthree}
z'_k \in U 
\end{equation}
for all sufficiently large
$k$. Furthermore, by using (\ref{equalone}) we find that
\begin{equation}\label{equal4}
\lim_k g_{F'}(z'_k)  = c\left(g_{F'}\right) = \lim_k
g_{F'}\left(z_k\right) \geq \frac{1}{2} .
\end{equation}
By combining (\ref{equaltwo}), (\ref{equalthree}) and (\ref{equal4})
we find that
$$
\varphi^{-1}\left(z'_k\right) \subseteq \bigcup_{i \in F'}
h_i^{-1}\left( \left]\frac{1}{4}, \infty\right[\right)
$$
for all large $k$. Since $y_k \in
\varphi^{-1}\left(z'_k\right)$, it follows that
$$
y_k \in \bigcup_{i \in F'} h_i^{-1}\left(\left]\frac{1}{4}, \infty\right[\right)
$$
for all large $k$. Hence there is an $i' \in F'$ such that $y_k \in
h_{i'}^{-1}\left(\left]\frac{1}{4}, \infty\right[\right)$ for
infinitely many $k$ which implies that 
$$
c''\left(h_{i'}\right) = \lim_k h_{i'}\left(y_k\right)  \geq
\frac{1}{4} .
$$
Hence $c'' \in W_{i'}$.
\end{proof}
\end{lemma}

\begin{cor}\label{cor1} $\# \psi^{-1}\left(\psi(c)\right) = c(m)$ for
  all $c \in \widehat{D}$.
\begin{proof} Using Lemma \ref{N} for the first equality we find that
$\# \psi^{-1}\left(\psi(c)\right) = \psi(c)(N) = c(N
  \circ \varphi) = c(m)$.
\end{proof}
\end{cor}

\begin{lemma}\label{yes!!}
  $1_{\Gamma_{\psi}(1,0)}^*f1_{\Gamma_{\psi}(1,0)}(c) = c\left(
    \sum_{z \in \varphi^{-1}(\cdot)} f(z)\right)$ for all $c \in
  \widehat{D}$ and all $f \in D$.
\begin{proof} Since both sides are continuous in $c$ and $\iota(X)$ is
  dense in $\widehat{D}$ it suffices to
  establish the identity when $c = c_x$ for some $x \in
  X$. It follows from Proposition \ref{prop} that we can apply
  (\ref{tauborineman}) with $\psi$ replacing $\varphi$ to conclude
  that 
$$
1_{\Gamma_{\psi}(1,0)}^*f1_{\Gamma_{\psi}(1,0)}(c_x) = \sum_{c' \in
  \psi^{-1}(c_x)} c'(f) .
$$
In comparison we have that
$$
c_x\left(
    \sum_{z \in \varphi^{-1}(\cdot)} f(z)\right) = \sum_{z \in \varphi^{-1}(x)} f(z) .
$$
So it remains only to show that
\begin{equation}\label{important}
\psi^{-1}(c_x) = \left\{ c_z : \ z \in \varphi^{-1}(x) \right\} .
\end{equation}
In fact, since the two sets have the same number of elements by Lemma
\ref{N}, it suffices to check that $\psi(c_z) = c_x$ when $z \in
\varphi^{-1}(x)$. This is straightforward:  $\psi(c_z)(f) = c_z(f\circ \varphi) =
f\left(\varphi(z)\right) = f(x) = c_x(f)$ for all $f \in D$.
\end{proof}
\end{lemma}

Note that (\ref{important}) means that
\begin{equation}\label{iotaX}
\psi^{-1}\left(\iota(X)\right) = \iota(X) .
\end{equation}

We can now adopt the proof of Theorem \ref{exeliso} to get the
following:

\begin{thm}\label{isowithlochomeo} There is a $*$-isomorphism
  $C^*_r\left(\Gamma_{\varphi}\right) \to
  C^*_r\left(\Gamma_{\psi}\right)$ which is the identity on
  $D_{\Gamma_{\varphi}}$ and takes the isometry $V_{\varphi} \in
  C^*_r\left(\Gamma_{\varphi}\right)$ to $V_{\psi} \in
  C^*_r\left(\Gamma_{\psi}\right)$.
\end{thm}
\begin{proof} We will appeal to Theorem \ref{exeliso} above and combine it
  with Corollary 7.2 of \cite{EV} for the
  existence of a $*$-homomorphism $C^*_r\left(\Gamma_{\varphi}\right) \to
  C^*_r\left(\Gamma_{\psi}\right)$ with the stated properties. We need
  therefore to check that
\begin{enumerate}
\item[1)] $V_{\psi}f = f \circ \varphi V_{\psi}$,
\item[2)] $c(V_{\psi}^*fV_{\psi}) = c\left( N(\cdot)^{-1} \sum_{z \in
      \varphi^{-1}(\cdot)} f(z)\right), \ c \in \widehat{D}$, and
\item[3)] $ 1 = \sum_{i=1}^k \left(b_i m\right)^{\frac{1}{2}}
  V_{\psi}V_{\psi}^* \left(b_i m\right)^{\frac{1}{2}}$ 
\end{enumerate}
where $f \in D$.
To check 1) note first that
$$
\left\{ \left(c_x,1, c_y\right)  : \ \varphi(x) = y\right\}
$$
is dense in $\Gamma_{\psi}(1,0)$. This follows from the density of $\iota(X)$ in $\widehat{D}$, the openness of
$\psi$ and (\ref{iotaX}). Since both sides of 1) are elements in $C_c\left(\Gamma_{\psi}(1,0)\right)$ it suffices therefore to check 1) on
elements of the form
$(c_x,1,c_y)$ with $\varphi(x) =y$ where it is easy: $V_{\psi}f\left(c_x,1,c_y\right) =
V_{\psi}\left(c_x,1,c_y\right)f\left(c_y\right) = f(\varphi(x)) = f\circ
\varphi v_{\psi}\left(c_x,1,c_y\right)$. The identity 2) is established
in a similar way: Since both sides are continuous functions on
$\widehat{D}$ it suffices to check it on elements from $\iota(X)$: 
\begin{equation*}
\begin{split}
&c_x\left( V_{\psi}^* fV_{\psi}\right) = \sum_{c' \in
  \psi^{-1}\left(c_x\right)} V_{\psi}^*\left(c_x, -1,
    c'\right)c'(f) V_{\psi}\left(c',1,c_x\right)\\
& = \sum_{c' \in
  \psi^{-1}\left(c_x\right)} \left(\#  \psi^{-1}\left(\psi(c')\right)\right)^{-1} c'(f) \\
&= \sum_{y \in
  \varphi^{-1}(x)} N(x)^{-1} f(y) \ \ \ \ \ \  \ \text{(by Corollary \ref{cor1}
  and (\ref{iotaX}))} \\
& = c_x \left( N(\cdot)^{-1} \sum_{z \in
      \varphi^{-1}(\cdot)} f(z)\right).   
\end{split}
\end{equation*}
To check 3) note that $\sum_{i=1}^k \left(b_i m\right)^{\frac{1}{2}}
  V_{\psi}V_{\psi}^* \left(b_i m\right)^{\frac{1}{2}} \in
  C_c\left(R(\psi)\right)$. Since elements of the form
  $\left(c_x,c_y\right)$ with $(x,y) \in R(\varphi)$ are dense in $R(\psi)$  
it suffices to show that for $(x,y) \in R(\varphi)$,
$$
\sum_{i=1}^k \left(b_i m\right)^{\frac{1}{2}}
  V_{\psi}V_{\psi}^* \left(b_i
    m\right)^{\frac{1}{2}}\left(c_x,c_y\right) = \begin{cases} 0 \ & \ \text{
    when $x \neq y$} \\ 1 \ & \ \text{when $x =y$.} \end{cases}
$$ 
So let $(x,y) \in R(\varphi)$. Then $\varphi(x) = \varphi(y)$ and we
find that
\begin{equation*}
\begin{split}
&\sum_{i=1}^k \left(b_i m\right)^{\frac{1}{2}}
  V_{\psi}V_{\psi}^* \left(b_i
    m\right)^{\frac{1}{2}}\left(c_x,c_y\right) \\ 
& = \sum_{i=1}^k
b_i(x)^{\frac{1}{2}}m(x)^{\frac{1}{2}}
m(x)^{-\frac{1}{2}}m(y)^{-\frac{1}{2}}
b_i(y)^{\frac{1}{2}}m(y)^{\frac{1}{2}} \ \ \ \ \ \ \ \ \ \ \ \ \ \ \ \text{(using Corollary
  \ref{cor1})} \\
& \\
&
= \begin{cases} 0 \ & \ \text{
    when $x \neq y$} \\ 1 \ & \ \text{when $x =y$} \end{cases}
\end{split}
\end{equation*}
since $\varphi$ is injective on $\supp b_i$ and
  $\sum_{i=1}^k b_i = 1$.
This establishes the existence of a $*$-homomorphism $\mu :
C^*_r\left(\Gamma_{\varphi}\right) \to
C^*_r\left(\Gamma_{\psi}\right)$ which is the identity on
$D_{\Gamma_{\varphi}}$ and takes $V_{\varphi}$ to $V_{\psi}$. The
injectivity of $\mu$ follows from the faithfulness of the
conditional expectation $P_{\Gamma_{\varphi}} :
C^*_r\left(\Gamma_{\varphi}\right) \to D_{\Gamma_{\varphi}}$ and the
observation that $P_{\Gamma_{\psi}} \circ \mu = P_{\Gamma_{\varphi}}$. And, finally, the
surjectivity of $\mu$ follows from the fact that
$C^*_r\left(\Gamma_{\psi}\right)$ is generated by $V_{\psi}$ and
$D_{\Gamma_{\psi}} = D_{\Gamma_{\varphi}}$. 
\end{proof}

By Theorem \ref{isowithlochomeo} we can identify
$C^*_r\left(\Gamma_{\varphi}\right)$ with
$C^*_r\left(\Gamma_{\psi}\right)$ and we will do that freely in the following.

\begin{remark}\label{isoremark} The isomorphism of Theorem
  \ref{isowithlochomeo} is clearly equivariant with respect to the
  gauge actions and it induces therefore an isomorphism between the
  corresponding fixed point algebras,
  $C^*_r\left(\Gamma_{\varphi}\right)^{\mathbb T}$ and
  $C^*_r\left(\Gamma_{\psi}\right)^{\mathbb T}$. Since $\psi$ is a
  local homeomorphism we have the equality
$C^*_r\left(\Gamma_{\psi}\right)^{\mathbb T} =
C^*_r\left(R_{\psi}\right).
$ 
Since there are subshifts $\sigma$ for which
$C^*_r\left(R_{\sigma}\right) \subsetneq
C^*_r\left(\Gamma_{\sigma}\right)^{\mathbb T}$
it follows that in general the isomorphism in Theorem
\ref{isowithlochomeo} does not take $C^*_r\left(R_{\varphi}\right)$
onto $C^*_r\left(R_{\psi}\right)$.
\end{remark}

\section{KMS states}\label{KMSsection}

Let $F : X \to \mathbb R$ be a realvalued function from $D$. Such a
function defines a continuous action $\alpha^F : \mathbb R  \to \Aut
C^*_r\left(\Gamma_{\varphi}\right)$ such that $\alpha^F_t(d) = d$ when
$d \in D_{\Gamma_{\varphi}}$ and $\alpha^F_t\left(V_{\varphi}\right) =
e^{iFt}V_{\varphi}$, cf. \cite{E}. 
The action $\alpha^F$ can also
be defined from the one-cocycle on $\Gamma_{\varphi}$ defined by $F$
as in the last line on page 2072 in
\cite{KR}, but the definition above allows us to combine Theorem
\ref{exeliso} with the work of Exel in \cite{E} to establish the connection
between the KMS states of $\alpha^F$ and the Borel probablity
measures on $\widehat{D}$ fixed by the dual of a Ruelle-type operator.

Let $\beta \in \mathbb R \backslash \{0\}$. A state $\omega$ on
$C^*_r\left(\Gamma_{\psi}\right)$ is \emph{a KMS state with inverse
  temperature $\beta$} for $\alpha^F$ (or just a $\beta$-KMS state for
short) when
\begin{equation}\label{KMS}
\omega(xy) = \omega\left(y \alpha^F_{i\beta}(x)\right)
\end{equation}
for all $\alpha^F$-analytic elements $x,y$ of
$C^*_r\left(\Gamma_{\varphi}\right)$. 

Let $\tau_{\lambda}, \lambda \in \mathbb T$, be the gauge action on
$C^*_r\left(\Gamma_{\psi}\right)$ (so that $\tau_t =
\alpha^F_{e^{it}}$ when $F$ is constant $1$) and let $P_{\Gamma_{\psi}} :
  C^*_r\left(\Gamma_{\psi}\right) \to D$ be the
  conditional expectation. Let $S(D)$ denote the set of states on
  $D$. When $\chi \in S(D)$ the composition $\chi \circ P_{\Gamma_{\psi}}$ is a
  state on $C^*_r\left(\Gamma_{\psi}\right)$. Note that $\chi \circ
  P_{\Gamma_{\psi}}$ is gauge-invariant since $P_{\Gamma_{\psi}} \circ
  \tau_{\lambda} = P_{\Gamma_{\psi}}$ for all $\lambda \in \mathbb
  T$.

Let $Q : C^*_r\left(\Gamma_{\psi}\right) \to
  C^*_r\left(R_{\Gamma_{\psi}}\right)$ be the conditional expectation 
$$
Q(x) = \int_{\mathbb T} \tau_{\lambda}(x) \ d\lambda .
$$

\begin{lemma}\label{KMSgauge} Let $\omega$ be a $\beta$-KMS state for
  $\alpha^F$. Then $\omega \circ Q$ is a gauge-invariant $\beta$-KMS state for
  $\alpha^F$.
\begin{proof} Let $x,y \in C^*_r\left(\Gamma_{\varphi}\right)$ be analytic for
$\alpha^F$. Since $\tau$ commutes with $\alpha^F$ we find that
\begin{equation*}
\begin{split}
& \omega \circ Q\left( xy\right) = \int_{\mathbb T}
\omega\left(\tau_{\lambda}(xy) \right) \ d \lambda \\
& = \int_{\mathbb T}
\omega\left(\tau_{\lambda}(x)\tau_{\lambda}(y) \right) \ d \lambda = \int_{\mathbb T}
\omega\left(\tau_{\lambda}(y)\alpha^F_{i\beta}\left(\tau_{\lambda}(x)\
  \right)\right) \ d \lambda \\
& = \int_{\mathbb T}
\omega\left(\tau_{\lambda}(y)\tau_{\lambda}(\alpha^F_{i\beta}(x)\right)
\ d \lambda =
\omega \circ Q\left(y\alpha^F_{i\beta}(x)\right) .
\end{split}
\end{equation*}

\end{proof}
\end{lemma}

For any $\beta \in \mathbb R$, define $L_{-\beta F} : D \to D$ such
that 
$$
L_{-\beta F}(g) (x) = \sum_{y \in \varphi^{-1}(x)} e^{-\beta F(y)}
g(y) .
$$

\begin{thm}\label{KMS} Let $\beta \in \mathbb R \backslash \{0\}$. The
  map $\chi \mapsto \chi \circ P_{\Gamma_{\psi}}$ is a bijection from the states $\chi \in S(D)$ which satisfy that
\begin{equation}\label{ooo}
\chi \circ L_{- \beta F} = \chi
\end{equation}
onto the gauge-invariant $\beta$-KMS states for $\alpha^F$.
\begin{proof} Consider first the case $\beta > 0$. By Proposition 9.2 and Section 11 in
  \cite{E} it suffices to show that any gauge-invariant $\beta$-KMS state $\omega$ of $\alpha^F$ factorizes through
  $P_{\Gamma_{\psi}}$, and this follows from Lemma 2.24 of \cite{Th}
  in the following way.
  Since $\omega$ is gauge-invariant we have that $\omega = \omega
  \circ Q$. Let $\left\{d_j\right\}$ be a partition of unity in $D$. Since $\alpha^F_{i\beta}\left(\sqrt{d_j}\right) =
  \sqrt{d_j}$ it follows from the KMS condition (\ref{KMS}) that
$\sum_j \omega\left(\sqrt{d_j}x\sqrt{d_j}\right) = \omega(x) $
for all $x \in C^*_r\left(\Gamma_{\varphi}\right)$. In particular,
$\omega(Q(x)) = \sum_j \omega\left(\sqrt{d_j}Q(x)\sqrt{d_j}\right)$
and hence
$\omega\left(P_{\Gamma_{\psi}}(Q(x))\right) = \omega(Q(x))$
 by Lemma 2.24 of \cite{Th} because $Q(x) \in C^*_r\left(R_{\psi}\right)$. Since $P_{\Gamma_{\psi}} \circ Q =
 P_{\Gamma_{\psi}}$ this shows that $\omega = \omega \circ
 P_{\Gamma_{\psi}}$ as desired.

The case $\beta < 0$ follows from the preceding case by observing that
$\omega$ is a $\beta$-KMS state for $\alpha^F$ if and only if $\omega$
is a $(-\beta)$-KMS state for $\alpha^{-F}$.
\end{proof}
\end{thm}

It follows from \cite{E} that every $\beta$-KMS state
is gauge invariant when $F$ is strictly
positive or strictly negative. This is not the case in general, but note that if there
is a $\beta$-KMS state for $\alpha^F$ then
there is also one which is gauge invariant by Lemma \ref{KMSgauge}.

We have deliberately omitted $\beta = 0$ as an admissable
$\beta$-value for KMS-states because they correspond to trace states and they
exist only in rather exceptional cases, e.g. when $\varphi$ has a
fixed point $x_0$ for which $\varphi^{-1}(x_0) =
\left\{x_0\right\}$.

\subsection{Bounds on the possible $\beta$-values}

Define $I_{\beta F} : D \to D$ such that
$$
I_{\beta F}(g)(x) = \frac{e^{\beta F(x)}}{m(x)} g \circ \varphi(x) .
$$
Then $L_{-\beta F } \circ I_{\beta F}(g) = g$ for all $g \in D$, so if $\chi \in
S(D)$ satisfies (\ref{ooo}) we find that
\begin{equation}\label{oo1}
\chi = \chi \circ I_{\beta F } .
\end{equation}
Thus $1 \in \text{Spectrum}\left(L_{-\beta F}^*\right) \cap
\text{Spectrum}\left(I_{\beta F}^*\right)$ when there is a state $\chi
\in S(D)$ for which (\ref{ooo})
holds. Let
$\rho(T)$ be the spectral radius of an operator $T$. Since 
$$
\text{Spectrum}\left(L_{-\beta F}^*\right) \cap
\text{Spectrum}\left(I_{\beta F}^*\right) =
\text{Spectrum}\left(L_{-\beta F}\right) \cap
\text{Spectrum}\left(I_{\beta F}\right),
$$ 
cf. \cite{DS}, we
find that
\begin{equation}\label{oo2}
1 \leq \rho\left(I_{\beta F}\right) 
\end{equation}
and
\begin{equation}\label{oo3}
1 \leq \rho\left(L_{-\beta F}\right)
\end{equation}
when (\ref{ooo}) holds.

To get the most out of these inequalities we consider a
non-invertible invariant $h_m$ which has been introduced for general
dynamical systems by M. Hurley in \cite{Hu} and developed further in
\cite{FFN}. For a locally injective map like the map
$\varphi$ we consider here, the invariant $h_m(\varphi)$ is simply
given by the formula
\begin{equation}\label{vaek}
h_m(\varphi) = \lim_{n \to \infty} \frac{1}{n} \log\left(\max_{x \in X} \# \varphi^{-n}(x)\right),
\end{equation}
cf. \cite{FFN}, or, alternatively, as
$$
h_m(\varphi) = \sup_{x \in X} \limsup_n \frac{1}{n} \log \#
\varphi^{-n}(x) ,
$$
cf. Corollary 2.4 of \cite{FFN}. For forward expansive maps, and hence
in particular for one-sided subshifts, $h_m$
equals the topological entropy $h$, but in general we only have the
inequality $h_m(\varphi) \leq h(\varphi)$. It can easily happen that
$h_m(\varphi) < h(\varphi)$ even when $\varphi$ is a
local homeomorphism.

The next lemma shows that for a locally injective surjection, as the
map $\varphi$ we consider, the invariant $h_m$ agrees with that of its
canonical local homeomorphic extension.

\begin{lemma}\label{equalh} $h_m\left(\psi\right) =
  h_m(\varphi)$.
\begin{proof} It follows from (\ref{iotaX}) that $\#
  \psi^{-k}\left(\iota(x)\right) = \#\varphi^{-k}(x)$ for all $x \in
  X$. Since $\# \psi^{-k}(c) = \sum_{c' \in \psi^{-k}(c)} 1$ depends
  continuously on $c \in \widehat{D}$ and $\iota(X)$ is dense in $\widehat{D}$, we
  conclude that $\max_{c \in \widehat{D}}\# \psi^{-k}(c) = \max_{x \in
    X}\# \varphi^{-k}(x)$. Hence $h_m\left(\psi\right) =
  h_m(\varphi)$, cf. (\ref{vaek}).
\end{proof}
\end{lemma}

In the following we let $M(X)$ denote the set of Borel probability
measures on $X$ and $M_{\varphi}(X)$ the subset of $M(X)$
consisting of the $\varphi$-invariant elements of $M(X)$. Similarly,
we let $M(\widehat{D})$ be the set of Borel probability
measures on $\widehat{D}$ and $M_{\psi}(\widehat{D})$ the
set of $\psi$-invariant elements in $M(\widehat{D})$.

\begin{lemma}\label{bounds2} Let $\beta \in \mathbb R$ and assume that
  there is a state $\chi \in S(D)$ such that (\ref{ooo}) holds. It
  follows that there are measures $\nu,\nu' \in
M_{\psi}(\widehat{D})$ such that
\begin{equation}\label{upperbound}
\beta  \int_{\widehat{D}} F \ d\nu \leq h_m(\varphi)
\end{equation}
and 
\begin{equation}\label{lowerbound}
\int_{\widehat{D}} \log \# \psi^{-1}(c)  \ d\nu'(c) \leq \beta \int_{\widehat{D}} F \ d\nu' .
\end{equation}
\begin{proof} Let $\delta > 0$. It follows from (\ref{oo3}) that $\rho\left(L_{-\beta
      F}\right) \geq 1$ which implies that
\begin{equation*}
\begin{split}
&-\delta \leq \frac{1}{k} \log
\left\|L^k_{-\beta F}(1)\right\|_{\infty} = \frac{1}{k} \log \left(
\sup_{c \in \widehat{D}}  \sum_{c' \in \psi^{-k}(c)} e^{-\beta
    \left(\sum_{j=0}^{k-1} F\left(\psi^j(c')\right)\right)}\right) \\
  \end{split}
\end{equation*}
for all large $k$. There is therefore, for each large $k$, a point
$c_k \in \widehat{D}$ such that
$$
-2\delta \leq  \frac{1}{k} \log \left(e^{
    \sum_{j=0}^{k-1} -\beta F\left(\psi^j(c_k)\right)} \sup_{c \in \widehat{D}}
      \# \psi^{-k}(c) \right) .
$$
Let $\nu$ be a weak* condensation point of the sequence
$$
\frac{1}{k} \sum_{j=0}^{k-1} \delta_{\psi^j\left(c_k\right)}
$$
in $M(\widehat{D})$. Then $\nu \in M_{\psi}(\widehat{D})$ by Theorem 6.9 of
\cite{W1} and
$$
\frac{1}{k}\sum_{j=0}^{k-1} -\beta F\left(\psi^j(c_k)\right) \leq  \int_{\widehat{D}}
-\beta F \ d\nu + \delta
$$
for infinitely many $k$. It follows that
$$
-2\delta \leq \frac{1}{k} \log
\sup_{c \in \widehat{D}} \#
  \psi^{-k}(c) + \int_{\widehat{D}} -\beta F \ d\nu +
      \delta
$$
for infinitely many $k$, and we conclude therefore that $0 \leq h_m(\psi) +
\int_{\widehat{D}} - \beta F d\nu$. Since $h_m\left(\psi\right) =
h_m(\varphi)$ by Lemma \ref{equalh} we get (\ref{upperbound}).

Similarly
it follows from (\ref{oo2}) that 
$$
1 \leq \lim_{k \to \infty} \sup_{c \in \widehat{D}}
\left(\frac{e^{\beta \sum_{j=0}^{k-1}
    F\left(\psi^j(c)\right)}}{\prod_{j=0}^{k-1} m\left(\psi^j(c)\right)}
\right)^{\frac{1}{k}} ,
$$
which implies that
$$
-\delta \leq \frac{1}{k} \log \left( \sup_{c \in \widehat{D}} e^{
    \sum_{j=0}^{k-1} \beta F\left(\psi^j(c)\right) - \log
    m\left(\psi^j(c)\right)}\right)
$$
for all large $k$. We can then work as before with $-\beta F$
replaced by $\beta F - \log m$ to produce the measure $\nu' \in M_{\psi}(\widehat{D})$ such that
$-2\delta \leq \int_{\widehat{D}} \beta F - \log m \ d \nu' +  \delta$. We omit
the repetition. Since
$\nu'$ is $\psi$-invariant we have that $\int_{\widehat{D}} \log m \ d\nu' =
\int_{\widehat{D}} \log \# \psi^{-1}(c) \ d\nu'(c)$. In this way we get (\ref{lowerbound}).

\end{proof}
\end{lemma}

When $H : X \to \mathbb R$ is a bounded realvalued function, set $A^{\varphi}_H(k) =
\inf_{x \in X} \sum_{j=0}^{k-1}
  H\left(\varphi^j(x)\right)$. Then $A^{\varphi}_H(k+n) \geq A^{\varphi}_H(k) +
A^{\varphi}_H(n)$ for all $n,k$ and we can set
$$
A^{\varphi}_H = \lim_{k \to \infty} \frac{A^{\varphi}_H(k)}{k} = \sup_n \frac{A^{\varphi}_H(n)}{n} .
$$
Similarly, we set $B^{\varphi}_H(k) = \sup_{x \in X} \sum_{j=0}^{k-1}
  H\left(\varphi^j(x)\right)$ and
$$
B^{\varphi}_H = \lim_{k \to \infty} \frac{B^{\varphi}_H(k)}{k} =
\inf_n \frac{B^{\varphi}_H(n)}{n} .
$$

\begin{prop}\label{bounds} When $\beta >0$ is the inverse temperature of
  a KMS state for $\alpha^F$ we have that
$A^{\varphi}_{\log m} \leq \beta B^{\varphi}_F $
and
$\beta A^{\varphi}_F \leq h_m(\varphi).$

When $\beta < 0$ is the inverse temperature of
  a KMS state for $\alpha^F$ we have that
$A^{\varphi}_{\log m} \leq \beta A^{\varphi}_F$
and
$\beta B^{\varphi}_F \leq h_m(\varphi)$.
  
\begin{proof} Let $\nu$ and $\nu'$ be the measures from Theorem
  \ref{bounds2}. When $\beta >0$ we find that
$$
h_m(\varphi) \geq \beta \int_{\widehat{D}} F \ d \nu = \beta \frac{1}{n} \int_{\widehat{D}}
\sum_{k=0}^{n-1} F
\circ \psi^k \ d\nu \geq \beta \frac{A^{\varphi}_F(n)}{n}  
$$
and
\begin{equation*}
\frac{A^{\varphi}_{\log m}(n)}{n} \leq \int_{\widehat{D}} \frac{1}{n} \sum_{k=0}^{n-1} \log
m\circ \psi^k \ d \nu' \leq \beta \frac{1}{n} \int_{\widehat{D}}
\sum_{k=0}^{n-1} F
\circ \psi^k \ d\nu' \leq \beta \frac{B^{\varphi}_F(n)}{n}
\end{equation*}
for all $n$. The two first inequalities of Theorem \ref{bounds}
follow from this. The case $\beta < 0$ is handled similarly.
 
\end{proof}
\end{prop}

\begin{cor}\label{nonexist1} Assume that $h_m(\varphi) = 0$. There are
  no KMS states for $\alpha^F$ unless $A^{\varphi}_F \leq 0 \leq B^{\varphi}_F$.
\end{cor}

\begin{lemma}\label{hagain} Assume that there is a $\beta$-KMS state
  for $\alpha^F$. It follows that there is a measure $\nu \in
  M_{\psi}(\widehat{D})$ such that
\begin{equation}\label{jul8}
\beta \int_{\widehat{D}} F \ d \mu \geq \limsup_n \frac{1}{n} \log \inf_{c \in
  \widehat{D}} \# \psi^{-n}(c) .
\end{equation}
\begin{proof} Let $\chi \in S(D)$ be a state such that $\chi \circ
  L_{-\beta F} = \chi$. Then
\begin{equation}\label{later}
\chi\left( \sum_{c' \in \psi^{-k}(\cdot)} e^{-\beta \sum_{j=0}^{k-1}F \circ
    \psi^j(c')} \right) = 1
\end{equation}
for all $k$ and hence
\begin{equation}\label{jul5}
\inf_{c \in \widehat{D}} \# \psi^{-k}(c) \chi\left( \frac{1}{\#
    \psi^{-k}( \cdot)} \sum_{c' \in \psi^{-k}(\cdot)} e^{-\beta
    \sum_{j=0}^{k-1} F \circ
    \psi^j(c')} \right) \leq 1
\end{equation}
for all $k \in \mathbb N$. Since $\log$ is concave we can apply
Jensen's inequality to the state $\mu$ on $D$ defined by
$$
\mu(g) =  \chi\left( \frac{1}{\#
    \psi^{-k}( \cdot)} \sum_{c' \in \psi^{-k}(\cdot)} g(c')  \right) .
$$
Then (\ref{jul5}) gives the estimate
\begin{equation}\label{jul7}
\log \inf_{c \in \widehat{D}}  \# \psi^{-k}(c)  -\beta \mu\left( \sum_{j=0}^{k-1} F
  \circ \psi^j\right)  \leq 0
\end{equation}
for all $k$. We can therefore choose a condensation point $\nu \in M_{\psi}(\widehat{D})$ of the sequence $\mu_k, k = 1,2, \dots$, where
$$
\mu_k(g) = \mu\left( \frac{1}{k} \sum_{j=0}^{k-1} g \circ \psi^j
\right) ,
$$
such that (\ref{jul8}) holds.
\end{proof}
\end{lemma}

\begin{thm}\label{ultbound} Assume that $F$ is continuous and that there is a $\beta$-KMS state
  for $\alpha^F$. Set 
$$
m = \lim_{n \to \infty} \frac{1}{n}
  \log \left(\min_{x \in X} \# \varphi^{-k}(x)\right)
$$
 and 
$$
M = \lim_{n \to \infty}
  \frac{1}{n} \log \left(\max_{x \in X} \# \varphi^{-k}(x)\right).
$$
There is then a $\varphi$-invariant Borel probability measure $\mu \in
M_{\varphi}(X)$ such that
$$
\beta \int_X F \ d \mu \in [m,M].
$$
\begin{proof} By Proposition \ref{bounds} and Lemma \ref{hagain} there
  are measures $\nu, \nu' \in M_{\psi}(\widehat{D})$ such
  that $\beta \int_{\widehat{D}} F \ d \nu \leq M$ and $m \leq \beta
  \int_{\widehat{D}} F \ d \nu'$. Since $F$ is continuous on $X$ by
  assumption we have that $F(c) = F(\pi(c))$ for all $c \in
  \widehat{D}$. It follows that with an appropriate convex combination
$$
\mu = s \nu \circ \pi^{-1} + (1-s) \nu' \circ \pi^{-1}
$$
we have that $m \leq \beta \int_X F \ d \mu \leq M$.  
\end{proof}
\end{thm}

\subsection{Existence of  KMS states}

While Proposition \ref{bounds} and Theorem \ref{ultbound} give upper and lower bounds on the
possible $\beta$-values of a KMS state for $\alpha^F$ they say nothing
about existence. This is where the work of Matsumoto, Watatani and
Yoshida, \cite{MWY}, and
Pinzari, Watatani and Yonetani, \cite{PWY}, comes in.

\begin{thm}\label{exist2} (cf. \cite{PWY} and \cite{MWY}) Let $B$ be a unital commutative
  $C^*$-algebra and $L : B \to B$ a positive linear operator with
  spectral radius $\rho(L)$. Then $\rho(L)$ is in the spectrum of $L$
  and there is
  a state $\omega \in S(B)$ such that $\omega \circ L =
  \rho(L)\omega$.
\begin{proof}  We adopt arguments from
  \cite{PWY} to show that $\rho(L)$ is in the spectrum of
  $L$ and then arguments from \cite{MWY}
  to produce the state $\omega$.

Recall that $\text{Spectrum}(L) = \text{Spectrum}(L^*)$, cf. \cite{DS}. By definition of $\rho(L)$ there is an element $z \in
\text{Spectrum}\left(L^*\right)$ with $|z| = \rho(L)$. Let
$\{z_n\}$ be a sequence of complex numbers such that $|z_n| > \rho(L)$ for
all $n$ and $\lim_n z_n
= z$. It follows then from the principle of uniform boundedness that
there is an element $\mu \in B^*$ such that
$$
\lim_{n \to \infty} \left\|R(z_n)\mu\right\| = \infty
$$
where $R(z) = \left(z - L^*\right)^{-1}$ is the
resolvent. Since $B^*$ is spanned by the states we may assume that
$\mu \in S(B)$. Since $|z_n| > \rho\left(L^*\right)$ the
resolvent $R(z_n)$ is given by the norm convergent Neumann series
$$
R(z_n) = \sum_{k=0}^{\infty} z_n^{-k-1}{L^*}^k .
$$
Since $\mu$ is a state and $L$ a positive
operator it follows that
$$
\left|R(z_n)\mu\right| \leq \sum_{k=0}^{\infty}
\left|z_n\right|^{-k-1}{L^*}^k\mu = R\left(|z_n|\right)\mu
$$
in $B^*$ where $\left|R(z_n)\mu\right|$ is the total variation measure of $R(z_n)\mu$.  Hence
$$
\left\|R(z_n)\mu\right\| \leq
\left\|R\left(\left|z_n\right|\right)\mu\right\|,
$$
and we conclude that $\lim_{n \to \infty}
\left\|R\left(\left|z_n\right|\right)\mu\right\| = \infty$, which
implies that $\rho(L) = \lim_{n \to \infty} |z_n|$ is in
$\text{Spectrum}\left(L^*\right) = \text{Spectrum}(L)$.

Set 
$$
\mu_n =
\frac{R\left(|z_n|\right)\mu}{\left\|R\left(|z_n|\right)\mu\right\|} .
$$
A glance at the Neumann series shows that $\mu_n$ is a state since
$L$ is positive. As
$$
\left(\rho(L) - L^*\right)\mu_n = \left(\rho(L) -|z_n|\right)\mu_n +
\left\|R\left(|z_n|\right)\mu\right\|^{-1}\mu 
$$
converges to $0$ in norm, any weak* condensation point $\omega$ of
$\{\mu_n\}$ will be a state such that $\omega \circ L = \rho(L)\omega$.

\end{proof}
\end{thm}

\begin{cor}\label{exist} Let $\beta \in \mathbb R \backslash\{0\}$ satisfy that the
  spectral radius $\rho\left(L_{-\beta F}\right)$ of $L_{-\beta F}$ is
  $1$. It follows that there is a gauge invariant $\beta$-KMS state
  for $\alpha^F$.
\begin{proof}  Combine Theorem \ref{exist2} with Theorem \ref{KMS}.
\end{proof}
\end{cor}

\begin{lemma}\label{Dexist} Let $A \subseteq \widehat{D}$ be a closed subset such
  that $\psi^{-1}(A) \subseteq A$. Assume that
$$
A^{\psi}_{F|_A} > 0.
$$

It follows that there are states $\omega,\nu,\nu' \in S(D)$ and a $ \beta \in
[0,\infty)$ such that 
\begin{enumerate}
\item[1)] $\nu \circ \psi = \nu$, $\nu ' \circ \psi = \nu'$,
\item[2)] $\omega(A) = \nu(A) = \nu'(A) =1$, 
\item[3)] $\beta \nu(F) \leq \lim_{n \to \infty} \frac{1}{n} \log
  \left( \max_{c
    \in A} \# \psi^{-k}(c) \right)  \leq \beta \nu'(F)$, and 
\item[4)] $\omega \circ L_{-\beta F} = \omega$.
\end{enumerate}
\begin{proof} Set $\delta = A^{\psi}_{F|_A} = \lim_n \left(\inf_{c \in A} \frac{1}{n} \sum_{k=0}^{n-1}
F\left(\psi^k(c)\right)\right)$. Since
  $\psi^{-1}(A) \subseteq A$ we can for any $t \in \mathbb R$ define a
  positive linear operator
  $L^A_{-t F} : C(A) \to C(A)$ such that
$$
L^A_{-t F}(g)(c) = \sum_{c' \in \psi^{-1}(c)} e^{-t F(c')}
g(c') .
$$
Then 
\begin{equation}\label{racommute}
L^A_{-t F} \circ r_A = r_A \circ L_{-t F} 
\end{equation}
where $r_A : D \to C(A)$ is the restriction map.
To estimate
the spectral radius of $L^A_{-t F}$ we observe that when $t
\geq 0$ we get the estimate  
\begin{equation*}
\begin{split}
&\sup_{c \in A}\left(L^A_{-t F}\right)^n(1)(c) = \sup_{c \in A} \sum_{c' \in \psi^{-n}(c)} e^{-t
  \sum_{k=0}^{n-1} F\left(\psi^k(c')\right)}  \\
&\leq \sup_{c \in A} \sum_{c' \in
  \psi^{-n}(c)} e^{-n\frac{t \delta}{2}}  \leq e^{-n\frac{t \delta}{2}}\sup_{c \in
  A}\# \psi^{-n}(c) 
\end{split}
\end{equation*}
for infinitely many $n$. It follows that
$$
\lim_{t \to \infty} \rho\left(L^A_{-t F}\right) = \lim_{t
  \to \infty} \lim_{n \to \infty} \left(\sup_{c \in A}\left(L^A_{-t
      F}\right)^n(1)(c)\right)^{\frac{1}{n}} = 0.
$$
On the other hand 
$$
\rho\left(L^A_0\right) = \lim_{n \to \infty} \left( \sup_{c \in A} \#
  \psi^{-n}(c)\right)^{\frac{1}{n}} \geq 1 .
$$
Since
$$
\left| \rho\left(L^A_{-t F}\right) - \rho\left(L^A_{-t'
      F}\right)\right| \leq \left|t - t'\right|
\left\|F\right\|_{\infty}
$$
for all $t,t' \in \mathbb R$, cf. Proposition 2.2 of \cite{ABL}, it follows that
$[0,\infty) \ni t \mapsto \rho\left(L^A_{-t F}\right)$ is
continuous. Hence the intermediate value theorem of calculus implies the
existence of a $\beta \in [0,\infty)$ such that $\rho\left(L^A_{-\beta
    F}\right) = 1$. Then Theorem \ref{exist2} implies the existence of
a state $\omega' \in S\left(C(A)\right)$ such that $\omega' \circ
L^A_{-\beta F} = \omega'$. Set $\omega = \omega' \circ r_A$ and note
that (\ref{racommute}) implies that $\omega \circ L_{-\beta F} =
\omega$. Since $\omega(f) = 0$ for all $f \in D$ with support in $X
\backslash A$ it follows that $\omega(A) = 1$.

To construct the $\psi$-invariant states $\nu$ and $\nu'$ let $\epsilon > 0$ and note that
\begin{equation}\label{cruxest}
\lim_{n \to \infty}  \frac{1}{n} \log \left( \sup_{c \in A} \sum_{c' \in
    \psi^{-n}(c)} 
  e^{-\beta \left(\sum_{k=0}^{n-1} F\left(\psi^k(c')\right)\right) }
\right) = 0 .
\end{equation}
For $n \in \mathbb N$ there are $c_n,c'_n
\in \psi^{-n}(A)$ such that
$$
\sum_{k=0}^{n-1} F\left(\psi^k(c_n)\right) = \inf_{c' \in
  \psi^{-n}(A)} \sum_{k=0}^{n-1} F\left(\psi^k(c')\right) \leq
\sup_{c' \in
  \psi^{-n}(A)} \sum_{k=0}^{n-1} F\left(\psi^k(c')\right) =\sum_{k=0}^{n-1} F\left(\psi^k(c'_n)\right).
$$
Then
\begin{equation}\label{jul2}
\begin{split}
&-\beta \frac{1}{n}\sum_{k=0}^{n-1} F\left(\psi^k(c'_n)\right) +
\frac{1}{n} \log \sup_{c \in A} \# \psi^{-n}(c)  \leq 0 \\
&\leq -\beta \frac{1}{n}\sum_{k=0}^{n-1} F\left(\psi^k(c_n)\right) +
\frac{1}{n} \log \sup_{c \in A} \# \psi^{-n}(c)
\end{split}
\end{equation}
for all $n$. Let $\nu$ and $\nu'$ be states of $D$ such that the corresponding
measures on $\widehat{D}$ are weak* condensation points of the sequences
$\frac{1}{n}\sum_{k=0}^{n-1}\delta_{\psi^k\left(c_{n}\right)}$ and $\frac{1}{n}\sum_{k=0}^{n-1}\delta_{\psi^k\left(c'_{n}\right)}, \  =
 1,2,3, \dots$,
respectively. Then 1) holds by Theorem 6.9 of \cite{W1} and $\nu(A) =
\nu'(A) =1$ since $A$ is closed and
$\psi^k\left(c_n\right), \psi^k\left(c'_n\right) \in A$ for all
$k,n$. The estimates 3) follow from (\ref{jul2}).

\end{proof}
\end{lemma}

\begin{thm}\label{existTHM} Assume that $h_m(\varphi) > 0$. 

\begin{enumerate}
\item[1)] If $A^{\varphi}_F > 0$ there is a
  $\beta$-KMS state for $\alpha^F$ such that $\beta A_F^{\varphi} \leq
  h_m(\varphi) \leq \beta B_F^{\varphi}$. 
\item[2)] If $B^{\varphi}_F < 0$ there is a
  $\beta$-KMS state for $\alpha^F$ such that $\beta B_F^{\varphi} \leq
  h_m(\varphi) \leq \beta A_F^{\varphi}$. 
\item[3)] When $F$ is continuous there is in both cases, 1) or 2), a
  $\varphi$-invariant Borel probability measure $\mu \in
  M_{\varphi}(X)$ such that 
\begin{equation}
\beta \int_X F \ d \mu = h_m(\varphi) .
\end{equation}
\end{enumerate}
 \begin{proof} 1) follows directly from Lemma
  \ref{Dexist} applied with $A = \widehat{D}$ and 2) follows
  by applying 1) to $-F$. 

3): Assume now that $F$ is continuous. Since we either have that $\beta A_F^{\varphi} \leq
  h_m(\varphi) \leq \beta B_F^{\varphi}$ or $\beta B_F^{\varphi} \leq
  h_m(\varphi) \leq \beta A_F^{\varphi}$ there is a sequence $n_1 < n_2 <
  \dots$ in $\mathbb N$ and points $x_i,y_i \in X$ such that
$$
h_m(\varphi) - \frac{1}{i} \leq \frac{1}{n_i} \beta\sum_{j=0}^{n_i -1} F \circ
\varphi^j\left(x_i\right)
$$
and
$$
\frac{1}{n_i}\beta \sum_{j=0}^{n_i -1} F \circ
\varphi^j\left(y_i\right) \leq h_m(\varphi) + \frac{1}{i}
$$
for all $i$. For each $i$ we can then find a number $s_i \in [0,1]$
such that 
\begin{equation}\label{thanks}
h_m(\varphi) - \frac{1}{i} \leq \frac{1}{n_i} \beta\sum_{j=0}^{n_i -1} \int_X F \circ
\varphi^j \ d \nu_i \leq h_m(\varphi) + \frac{1}{i}, 
\end{equation}
where $\nu_i = s_i \delta_{x_i} + (1-s_i) \delta_{y_i}$. Any weak*
condensation point of the sequence 
$$
\frac{1}{n_i} \sum_{j=0}^{n_i-1} \nu_i \circ \varphi^{-j} 
$$
will be $\varphi$-invariant by Theorem 6.9 of \cite{W1} and $\beta
\int_X F  \ d \mu = h_m(\varphi) $ thanks to (\ref{thanks}).

\end{proof}
\end{thm}

\begin{cor}\label{compare} Assume that $F$ is continuous and either
  strictly positive or strictly negative. There is no KMS-state for
  $\alpha^F$ if $h_m(\varphi) = 0$. If $h_m(\varphi) > 0$ there is a
  $\beta$-KMS-state for $\alpha^F$ such that
$$
\beta = \frac{h_m(\varphi)}{\int_X F \ d \mu}
$$
for some $\mu \in M_{\varphi}(X)$. 
\end{cor}
\begin{proof} The first assertion follows from Corollary \ref{cor1}
  and the second from Theorem \ref{existTHM}.
\end{proof}

\begin{example}\label{constto1} Assume that $\varphi : X \to X$ is
  uniformly $n$-to-$1$, i.e. that $\# \varphi^{-1}(x) = n$ for all $x
  \in X$. Note that $n \geq 2$ since we assume that $\varphi$ is not
  injective. Then $h_m(\varphi) = \log n$ and it follows from Theorem
  \ref{existTHM} and Theorem \ref{ultbound} that there is exactly one
  $\beta$ such that the gauge action on
  $C^*_r\left(\Gamma_{\varphi}\right)$ has a $\beta$-KMS state, namely
  $\beta = \log n$. In many cases $\log n$ is also the topological
  entropy $h(\varphi)$. This is for example the case when $\varphi$ is
  an affine map on $\mathbb T^k$. To see that in general $\log n$ is
  smaller than the topological entropy let $f : Y \to Y$ be an
  arbitrary homeomorphism of a compact metric space $Y$. Then $\varphi
  \times f : X \times Y \to X\times Y$ is also locally injective and
  $n$-to-$1$. In particular
  $h_m\left(\varphi \times f\right) = \log n$ while the topological
  entropy is $h(\varphi) + h(f)$ which can be any number $\geq \log n$.
  \end{example}

\end{document}